\documentclass[12pt]{amsart}
\usepackage{latexsym}
\usepackage{amssymb}
\newtheorem{theorem}{Theorem}[section]
\newtheorem{lemma}[theorem]{Lemma}
\newtheorem{proposition}[theorem]{Proposition}
\newtheorem{corollary}[theorem]{Corollary}
\newtheorem{remark}[theorem]{Remark}

\numberwithin{equation}{section}
\def\gd{\delta}
\def\gs{\sigma}

\def\bfap{{\bar \fa^+}}

\def\fP{{\mathfrak P}}
\def\bare{{\bar e}}
\def\Hom{{\rm Hom}}
\def\Tr{{\rm Tr}}
\def\after{\circ}
\def\inp#1#2{\langle\,#1, #2\, \rangle}
\def\fm{{\mathfrak m}}
\def\fad{\fa^*}
\def\reg{{\rm reg}}

\def\gS{\Sigma}
\def\ga{\alpha}
\def\gb{\beta}
\def\Cartan{\theta}
\def\fp{{\mathfrak p}}
\def\dotvar{\,\cdot\,}
\def\cF{{\mathcal F}}
\def\sl{sl}
\def\ft{{\mathfrak t}}
\def\fz{{\mathfrak z}}
\def\Re{{\rm Re}\,}
\def\Im{{\rm Im}\,}
\def\bilform{{\rm B}}
\def\End{{\rm End}}
\def\Vect{{\rm Vect}}
\def\ad{{\rm ad}\,}
\def\fa{{\mathfrak a}}
\def\fk{{\mathfrak k}}
\def\fg{{\mathfrak g}}

\def\fn{{\mathfrak n}}
\def\gf{{\varphi}}
\def\ga{\alpha}
\def\gl{\lambda}

\def\om{\omega}
\def\R{{\mathbb R}}
\def\C{{\mathbb C}}

\def\Del{\Delta}
\def \cC{{\mathcal C}}
\def\pr{{\rm pr}}
\def\Ad{{\rm Ad}}
\def\cO{{\mathcal O}}

\def\un1{\underline}
\def\Om{\Omega}

\begin{document}

\baselineskip=16pt

\title{Symplectic geometry of semisimple orbits}

\author[H. Azad]{Hassan Azad}

\address{Department of Mathematics, Lahore University of
Management Sciences, Pakistan; Department of Mathematics and
Statistics, King Fahd University, Saudi Arabia}

\email{hassanaz@kfupm.edu.sa}

\author[E. P. van den Ban]{Erik van den Ban}

\address{Mathematisch Instituut, Universiteit Utrecht,
PO Box 80 010, 3508 TA Utrecht, The Netherlands}

\email{E.P.vandenBan@uu.nl}

\author[I. Biswas]{Indranil Biswas}

\address{School of Mathematics, Tata Institute of Fundamental
Research, Homi Bhabha Road, Bombay 400005, India}

\email{indranil@math.tifr.res.in}

\date{}

\begin{abstract}

Let $G$ be a complex semisimple
group, $T\subset G$ a maximal torus and $B$ a Borel subgroup of $G$ containing
$T.$
Let $\Omega$ be the Kostant--Kirillov holomorphic
symplectic structure on the adjoint orbit ${\mathcal O} = \Ad(G)c \simeq G/Z(c)$,
where $c\,\in\, {\rm Lie}(T)$, and $Z(c)$ is
the centralizer of $c$ in $G$.
We prove that the real symplectic form $\text{Re}\,\Omega$
(respectively, $\text{Im}\,\Omega$)
on ${\mathcal O}$ is exact if and only if all the
eigenvalues $\ad(c)$ are real (respectively, purely imaginary).
Furthermore, each of these real symplectic manifolds
is symplectomorphic to the cotangent
bundle of the partial flag manifold $G/Z(c)B,$  equipped with the
Liouville symplectic form. The latter result is generalized to
hyperbolic adjoint orbits in a real semisimple Lie algebra.

\end{abstract}

\maketitle

\section{Introduction}

This work grew out of attempts to understand the following
theorem of Arnold \cite[p. 100, Theorem 1]{Ar}.

\begin{theorem}[\cite{Ar}]\label{ar.th.}
Let $\Omega$ be the standard complex symplectic structure on a
regular coadjoint orbit of the group ${\rm SL}(n+1, {\mathbb C})$.
This orbit, equipped with the real symplectic structure ${\rm Im}
(\Omega)$, is isomorphic to the total space of the cotangent bundle
of the variety parametrizing the complete flags in ${\mathbb C}^{n+1}$,
equipped with the standard Liouville symplectic structure on it,
if and only if all the eigenvalues of some (and hence any) matrix
in the orbit are real.
\end{theorem}

A proof of this theorem is outlined in \cite[p. 100--101]{Ar}. The
assertion about the equivalence of the above mentioned
symplectic structure $\text{Im}(\Omega)$ with the one on total space of the cotangent bundle of
the flag variety is made in lines 13--15 of \cite[p. 101]{Ar}.
Apparently,
the regular coadjoint orbit is identified with an adjoint orbit in
$\sl(n+1, \C)$ through the non-degenerate bilinear form $(X,Y) \mapsto {\rm Tr}(XY),$
so that it makes sense to speak of eigenvalues of matrices in the orbit.

Arnold's result may be reformulated in terms of the theory of semisimple Lie groups.
In the present paper we will state this reformulation and prove a generalization
of it.

Let $G$ be a connected complex semisimple Lie group. Its Lie algebra, which will be denoted by  $\fg,$ comes
equipped with the Killing form $B,$ which is an $\Ad(G)$-invariant symmetric non-degenerate
bilinear form. Given an element $c \in \fg,$ we denote by $B(c)$ the complex linear functional
on $\fg$ defined by $X \mapsto B(c,X).$ Accordingly, the Killing form is viewed as
a $G$-equivariant linear isomorphism
$$
B:  \fg \buildrel\simeq\over\longrightarrow \fg^*.
$$
Unless specified otherwise, we will use $B$ to identify $\fg$ with $\fg^*.$
In particular, by pull-back under $B$ of the canonical Kostant-Liouville holomorphic symplectic
form on any coadjoint orbit $\cO \subset \fg^*$ may be viewed as a holomorphic
symplectic form on the associated adjoint orbit $B^{-1}(\cO).$

If $c \in \fg,$ then by $\ad(c)$ we denote the endomorphism $Y \mapsto [c,Y]$ of $\fg.$
The element $c$ is called semisimple if and only if $\ad(c)$ diagonalizes. Equivalently,
this means that $c$ is contained in the Lie algebra of a maximal torus (or Cartan subgroup) $T$ of $G.$
The centralizer of $c$ in $G$ is denoted by $Z(c).$ If $c$ is semisimple,
then $Z(c)$ is known to be the Levi component of a parabolic subgroup $P$ of $G.$
In fact, one may take $P = Z(c)B,$ where $B$ is a Borel subgroup containing
a maximal torus which contains
$c.$ We will prove the following generalization of Arnold's result.

\begin{theorem}\label{t2}
Let $G$ be a connected complex semisimple group, and let $c$ be a semisimple element of its
Lie algebra $\fg.$ Let $\Omega$ be the Kostant--Kirillov holomorphic symplectic
form on the orbit $\cO = \Ad(G)c \simeq G/Z(c).$
Then the  real and imaginary parts ${\rm Re}\,\Om$ and ${\rm Im}\,\Om$ are real symplectic forms on $\cO.$
Moreover, the following hold.
\begin{enumerate}
\item[{\rm (a)}]
The form ${\rm Re}\,\Om$
(respectively, ${\rm Im}\,\Om$)
on ${\mathcal O}$ is exact
if and only if all eigenvalues
of $\ad (c)$ are real
(respectively, purely imaginary).
\item[{\rm (b)}]
In either case, these symplectic manifolds with exact real symplectic forms are
symplectomorphic to the total space of the cotangent bundle
of $G/P,$ equipped with the Liouville symplectic form,
where $P$ is any parabolic subgroup of $G$ with Levi component $Z(c).$
\end{enumerate}
\end{theorem}

In fact, we will prove a refinement of assertion (b) in the more general context of
a real hyperbolic adjoint orbit of a real semisimple Lie group;
see Theorems \ref{t: complex Lagrangian fibration} and \ref{t: real Lagrangian fibration}.

Here are a few words about our interpretation of the above mentioned result of Arnold.

Set $G\, =\, \text{SL}(n+1, {\mathbb C})$, and let $T\, \subset\, G$ be the subgroup of diagonal matrices.
For any ${c}\, \in\, \text{Lie}(T)$ with distinct eigenvalues we have
$Z(c) = T,$ so that
the adjoint orbit of $c$ can be identified with $G/T$. Let $c_i$ denote
the $i$-th diagonal entry of $c.$
The eigenvalues
of $\ad(c)$ are all the  numbers of the form $c_i - c_j,$ with $1 \leq i,j \leq n.$
As $\sum_j c_j = 0,$ it follows that the eigenvalues of $c$ are all real if and only
if those of $\ad(c)$ are.

The group $G$ naturally acts on the manifold $\cF$ of full flags in $\C^{n+1}.$ The stabilizer
of the standard flag $\C\subset\C^2 \subset \cdots \subset \C^{n}$ is
the subgroup $B\, \subset\, G$  of upper triangular matrices. Consequently, $\cF \simeq G/B,$
as $G$-manifolds. Arnold's result asserts that $\Ad(G)c \simeq T^*\cF \simeq T^*(G/B)$ as real
symplectic manifolds.
In the present set-up, our generalization concerns the analogue for an arbitrary
diagonal matrix $c$ and the associated partial flag manifold $G/P.$

As $G/T\,\simeq \, G\times_B (B/T)$,
the natural projection
$$
\psi\, : \, G/T\, \longrightarrow\, G/B
$$
makes $G/T$ a fiber bundle over the full flag manifold $G/B$; its
fibers are translates of $B/T$.  Since
$G/B$ is a complete variety, and $G/T \simeq \cO$ an affine variety, the
bundle $\psi$ does not admit any holomorphic sections.

On the other hand, let $K\, =\, \text{SU}(n+1)$. Then the natural map $j: K/K\cap T \to G/T$
determines a real analytic section of $\psi.$ Indeed, since $G = KB$ and $K \cap B = K \cap T,$
the map $ K/K \cap T \to G/B$ is a real analytic diffeomorphism. Composing its inverse with $j$ we obtain
a section $s: G/B \to G/T.$ Moreover,
$$
G/T\,\simeq \, G\times_B (B/T)\,=\, KB\times_B(B/T)\, \simeq\,
K\times_{K\cap B} (B/T)\,\simeq K \times_{K\cap T} R_u(B),
$$
where $R_u(B)$ denotes the unipotent radical of $B.$ This unipotent radical has the structure of
a complex linear space on which the adjoint action of $T$ linearizes. Therefore,
the last isomorphism realizes $G/T$ as a real analytic vector bundle
over $K/(K\cap T) \simeq G/B$ with $s$ corresponding to the zero section.
This real analytic vector bundle is
in fact isomorphic to the cotangent bundle of $K/(K \cap T)$. It follows that
the  inclusion $K/(K \cap T) \to G/T,$ and hence the real analytic section $s: G/B \to G/T,$
induces an isomorphism on the cohomology algebras of these spaces.
Hence, one can decide whether a given closed
differential
two-form on $G/T$ is exact from its restriction to $K/(K\cap T).$
This is roughly a translation in group theoretic
terms of \cite[p. 100--101]{Ar}. The generalization of this argument to our
more general setting is worked out in the next section and leads to part (a) of Theorem \ref{t2}.

Assertion (b) in Theorem \ref{t2} is based on the crucial observation that
the fibration $\psi: G/Z(c) \to G/P$ has Lagrangian fibers and that $K/K\cap Z(c) \hookrightarrow G/T$
defines a Lagrangian section. This implies the existence of commuting vertical vector fields
on the bundle $\psi$ and is enough to establish the existence of
a local symplectic isomorphism along the
section $K/K\cap Z(c);$ see Section \ref{sect3}. This argument is indicated in \cite{Ar},
but an argument for the existence of a globally defined symplectic isomorphism seems to
be lacking.

We prove the existence of such a global symplectomorphism
in Section \ref{sect4} by showing that the mentioned vertical vector
fields have complete flows which can be used to construct global coordinates along the fibers
of $\psi.$ Moreover, we give this argument of integration
in the more general setting of real hyperbolic adjoint orbits for a
real semisimple Lie group.

The above mentioned commuting vector fields are used to construct
a $K$--equivariant diffeomorphism
$$
\phi\, :\, K\times_{K\cap Z(c)}(\fg/{\rm Lie}(P))^*\, \longrightarrow
\, G/Z(c) \, .
$$
The pull back of $\text{Re}(\Omega)$ --- in the notation of
Theorem \ref{t2} --- is the Liouville form on
$K\times_{K\cap Z(c)}(\fg/{\rm Lie}(P))^*$ identified with $T^*(G/P)$.

{\bf Acknowledgement:\ } One of us (EvdB) would like to thank Hans Duistermaat for
a helpful discussion on symplectic geometry.

\section{Complex semisimple orbits}\label{sec1}
We will recall some generalities concerning the Kostant-Kirillov symplectic form,
after fixing the notation.
At first we assume that $G$ is a connected Lie group over the base field $k,$ which is
either $\R$ or $\C.$ Let $\eta$ be an element of $\fg^*,$ the $k$-linear dual of $\fg.$
Let $Z(\eta)$ denote the stabilizer of $\eta$ in $G,$
and let $\fz(\eta)$ be the Lie algebra of $Z(\eta).$
The map $x \mapsto \eta\after \Ad(x)^{-1}$
induces a $G$-equivariant diffeomorphism from $G/Z(\eta)$ onto the coadjoint orbit
$\cO  = \cO_\eta \subset \fg^*$ through $\eta.$

The Kostant-Kirillov form $\Omega = \Omega_\eta$ on $\cO$
is defined as follows.
The action of $G$ on $\cO$ gives rise to a Lie algebra homomorphism
from $\fg$ to the space $\Vect(\cO)$ of vector fields on $\cO.$ Given $X \in \fg,$ the
associated vector field $\bar X$ on $\cO$ is given by $\bar X_\xi =: - \xi \after \ad X  \in T_\xi \cO \subset \fg^*,$
where $\xi \in \cO.$
We agree to write $X_\xi$ for $\bar X_\xi$ and note that the map $X \mapsto X_\xi$ descends
to an isomorphism from $\fg/\fz(\xi)$ onto $T_\xi\cO.$ The two-form $\Omega$  on $\cO$
is given by the formula
\begin{equation}\label{e: defi Omega}
\Omega_\xi(X_\xi\, , Y_\xi)\, =\, \xi([X\, ,Y]),
\end{equation}
where $\xi\in \cO$ and $X,Y \in \fg.$ Here we note  that the expression
on the right-hand side of (\ref{e: defi Omega}) depends on $X$ and $Y$ through their images in
$\fg/\fz(\xi),$ so that $\Omega$ is a well-defined.
The form $\Omega$ is $G$-invariant. Moreover, it is readily seen
to be closed and non-degenerate at the point $\xi,$ hence it is a symplectic form. See
\cite[p. 6]{Ki}.
Note that if $k = \C,$ then $\cO$ is a complex submanifold of $\fg^*,$ and $\Omega$
is a holomorphic symplectic form.

Via the natural diffeomorphism $G/Z(\eta) \to \cO$ the form $\Omega$ may be pulled-back
to a form on $G/Z(\eta).$ The resulting form, also denoted by $\Omega,$ is the unique
$G$-invariant two-form which at the
element $\bar e: = eG(\eta)$ is given by the formula
\begin{equation}
\label{e: Omega at e}
\Omega_{\bar e} (X_{\bar e}, Y_{\bar e}) = \eta([X,Y]), \qquad (X,Y \in \fg).
\end{equation}
We now assume that $G$ is a semisimple connected Lie group over $k$,
so that the Killing form $B(X,Y) = \Tr(\ad(X)\ad(Y)$ is non-degenerate on $\fg.$
The form $B$ is $G$-invariant and symmetric. Hence it induces
a $G$-equivariant linear isomorphism $\fg \to \fg^*$ that maps adjoint orbits
diffeomorphically and $G$-equivariantly onto coadjoint orbits. Let $\xi \in \fg^*$ and let $c = c_\xi
= B^{-1}(\xi).$ This means that
\begin{equation}
\label{e: eta and c}
\eta(Y) = B(c, Y), \qquad (Y \in \fg).
\end{equation}
Then $Z(\eta)$ coincides with $Z(c),$ the centralizer of $c$ in $G.$ Via pull back under $B,$
the form $\Omega$ may be realized as a form on the adjoint orbit $\Ad(G)c.$

In the rest of this section we assume that $k = \C,$ so that $G$ is a connected complex semisimple
Lie group. We assume that $\eta \in  \fg^*$ is such that $c = c_\eta$ is semisimple, i.e., the endomorphism
$\ad(c) \in \End(\fg)$ given by $X \mapsto [c, X]$ is diagonalizable. Equivalently, this
means that $c$ is contained in the Lie algebra of a maximal torus in $G.$

We fix a maximal torus $T$ of $G$ whose Lie algebra contains $c,$ and in addition
a maximal compact subgroup $K$ of $G$
for which $K \cap T$ is a maximal torus. Writing $\mathfrak k$  for the (real)
Lie algebra of $K,$ we have
\begin{equation}
\label{e: real form g}
{\mathfrak g}\, =\, {\mathfrak k}\oplus \sqrt{-1}
\cdot{\mathfrak k}
\end{equation}
as a direct sum of real linear spaces.
In particular, $\mathfrak k$ is a real form
of $\fg.$ The associated conjugation map $\theta: \fg \to \fg$ is called the Cartan-involution
associated with $K.$

\begin{lemma}\label{l: real symplectic forms}
With notation as above, let $\Omega$ be the holomorphic
Kostant-Kirillov symplectic form on $G/Z(c) = G/Z(\eta).$ Then both
$\Re \Omega$ and $\Im \Omega$ are real symplectic forms on $G/Z(c).$
\end{lemma}

\begin{proof}
We will write $\fg_\R$ for $\fg,$ viewed as a real Lie algebra. Accordingly, we put
$\fg_\R^*$ for the real linear dual of $\fg_\R.$ Then $\fg_\R^* = \Hom_\R(\fg, \R).$
Both $\Re \eta$ and $\Im \eta$ belong to $\fg_\R^*.$ Let $Z(\Re \eta)$ and $Z(\Im \eta)$
be the stabilizers of these elements for the coadjoint action for $G,$ viewed as
a real Lie group. We claim that
$$
Z(\eta) = Z(\Re \eta) = Z(\Im \eta).
$$
Indeed, this is seen as follows.
Let $J$ denote the linear automorphism of $\fg$  given by $X \mapsto \sqrt{-1}\cdot X.$
Pull-back by $J$ induces the real linear automorphism  $J^*$ of $\fg_\R^*$
given by $\xi \mapsto \xi \after J.$
As $G$ is a complex Lie group, the adjoint action of $G$ on $\fg$ commutes with $J.$ Therefore,
the coadjoint action of $G$ on $\fg^*_\R$ commutes with $J^*.$ It follows that
$Z(J^* \xi) = Z(\xi)$ for all $\xi \in \fg_\R^*.$ Now  $J^* \Re \eta = \Re (i \eta) = - \Im \eta,$
from which we see that $Z(\Re \eta) = Z(\Im \eta).$ Since $Z(\eta) = Z(\Re \eta) \cap Z(\Im \eta),$ the claim follows.

We now observe that $\Re \Omega$ is the unique $G$-invariant two-form on $G/Z(\eta) = G/Z(\Re \eta)$
given by $\Re \Omega_{\bar e}(X_\bare, Y_\bare) = [\Re \eta](X,Y).$ This implies that
$\Re \Omega$ is just the Kostant-Kirillov form associated with the coadjoint orbit
through $\Re \eta$ in $\fg_\R,$ with $G$ viewed as a real semisimple Lie group.
Likewise, $\Im \Omega$ is the form associated with the coadjoint orbit through $\Im \eta$ in $\fg_\R^*.$
\end{proof}

In the rest of this section we will prove the following theorem.

\begin{theorem}\label{t2a}
Let
$c\,\in\, {\rm Lie}(T),$ let $\eta = B(c, \dotvar) \in \fg^*$ and let
$\Om$ be the holomorphic Kostant-Kirillov symplectic form on
$G/Z(c)$ defined by (\ref{e: Omega at e}).

The real symplectic form ${\rm Re}\,\Om$
(respectively, ${\rm Im}\,\Om$)
on $G/Z(c)$ is exact if and only if all eigenvalues of $\ad(c)$
are real (respectively, purely imaginary).
\end{theorem}

We
will prove Theorem \ref{t2a} through a number of lemmas.

\begin{lemma}\label{l: centralizer stable}
The centralizer $\fz(c)$ is stable under $\theta.$ Equivalently, $\fk \cap \fz(c)$ is a real
form of $\fz(c).$
\end{lemma}

\begin{proof}
Write $\ft$ for the Lie algebra of the maximal torus $T.$ Since $T \cap K$ is a maximal
torus of $K,$ we have
$$
\ft = \ft \cap \fk + \sqrt{-1}\,( \ft \cap \fk).
$$
Accordingly, we write $c\, =\, a+\sqrt{-1}b$, where $a$ and $b$ belong to $\ft \cap \fk.$
Fix a positive definite $K$-invariant Hermitian inner product $\inp{\dotvar}{\dotvar}$ on $\fg.$
Then $\ad(a)$ is anti-Hermitian, hence diagonalizable with purely imaginary eigenvalues.
Similarly, $\ad (\sqrt{-1} b)$ is diagonalizable with real eigenvalues. Since
$\ad(a)$ and $\ad(\sqrt{-1} b)$ commute, they allow a simultaneous diagonalization.
{}From this we see that $\ker (c)$ is the intersection of $\ker \ad(a)$ and $\ker \ad(b).$
Since both $a$ and $b$ are $\theta$-stable, it follows that $\fz(c) = \ker \ad (c)$
is $\theta$-stable.
\end{proof}

If $g\, \in\, G$ centralizes $c$, then $g$ also centralizes
the one--parameter subgroup $\{\exp(tc)\mid t\in {\mathbb C}\}$
of $G$. The closure of this one--parameter subgroup will be
denoted by $S$. Clearly $g$ centralizes $S$. In other words,
we have $Z(c) \, =\, Z(S)$.

It is well known that the centralizers
of tori are connected reductive. More precisely, $Z(c) = Z(S)$ is the Levi
complement of a parabolic subgroup of $G$ \cite[p. 26, Proposition
1.22]{DM}, \cite{BH}.

Fix a simple system $\Delta_1$ of roots of the reductive group $Z(c)$ relative to the maximal torus $T$ and extend it
to a simple system $\Delta$ of roots of $G$ relative to the same maximal torus.
Let $B$ be the Borel subgroup of $G$ defined by the simple system
of roots $\Del$.  Then $P = Z(c) B$ is a parabolic
subgroup of $G.$ Its Levi-complement is $Z(c)$, and its unipotent radical $R_u(P)$
is given by the roots in $B$ whose supports are not contained in $\Del_1$. So $P\, =\, Z(c) R_u(P)$, and
$G\, =\, KB\, =\, KP$. We agree to write $Z_K(c)$ for $ K \cap Z(c),$ the centralizer of $c$ in $K.$

\begin{lemma}\label{le1}
The manifold $G/Z(c)$ is real analytically a vector bundle
over  $K/Z_K(c)$.
\end{lemma}

\begin{proof}
This is a consequence of basic results of Mostow \cite{Mo}.
A direct argument is as follows.

The exponential map induces a holomorphic diffeomorphism from ${\rm Lie}(R_u(P))$ onto $R_u(P) \simeq P/Z(c).$
Accordingly, we equip $P/Z(c)$ with the structure of a complex vector space. As $k \exp X Z(c) = \exp \Ad(k)X  Z(c)$
for $X \in R_u(P)$ and $k \in K \cap P = K \cap Z(c),$ the action of $K \cap P$ on $P/Z(c)$
by left translation is linear for this structure. Accordingly,
$$
K \times_{K\cap P} P/Z(c) \to K/K \cap P = K /Z_K(c)
$$
has the structure of a real analytic vector bundle over $K/Z_K(c).$

The multiplication map induces a surjective and submersive real analytic map
$K \times P \to G,$ which factors
to a submersive real analytic map
$$
K \times_{K \cap P} P \to G.
$$
This map is clearly injective, hence a real analytic diffeomorphism.
Therefore, the induced map
$$
K \times_{K \cap P} P/Z(c) \to G/Z(c)
$$
is a real analytic diffeomorphism as well. It realizes $G/Z(c)$ as a real analytic vector bundle over $K/Z_K(c).$
\end{proof}

\begin{lemma}\label{le2}
Let $\Omega$ be the holomorphic
symplectic form on $G/Z(c)$ defined in (\ref{e: Omega at e}).
The $K$--orbit
through $\bare \,=\, eZ(c)$ is Lagrangian
relative to ${\rm Im}\,\Om$ (respectively,
${\rm Re}\,\Om$) if and only if
all the eigenvalues of $\ad(c)$ are
purely imaginary (respectively, real).
\end{lemma}

\begin{proof}
In Lemma \ref{l: real symplectic forms} we established that $\Re \Omega$ and $\Im \Omega$
are real symplectic forms on $G/Z(c).$ It follows from Lemma \ref{l: centralizer stable}
that $K\bare \simeq K/Z_K(c)$ is a real form for $G/Z(c).$ In particular, $K \bare$ has
half the real dimension of $G/Z(c).$ Hence it suffices to establish
the above assertion with the word Lagrangian replaced by isotropic.

It follows from (\ref{e: Omega at e}) combined with (\ref{e: eta and c}) that the form $\Im \Omega$
is at $\bare = eZ(c)$ given by
$$
\text{Im}(\Omega_{\bare}(X_{\bare}\, , Y_{\bare}))\, =\,\Im B(c\,, \,[X\, ,Y]), \qquad (X,Y \in \fg).
$$
We write $c = a + \sqrt{-1} b$ with $a, b \in \ft \cap \fk,$ as in the proof of Lemma
\ref{l: centralizer stable}.
Since $\bilform$ is real-valued on $\fk,$ it follows that
$$
\text{Im}(\Omega_{\bare}(X_{\bare}\, , Y_{\bare}))\, =\,\bilform(b\,,\,[X\, ,Y]))
$$
for all $X\, , Y\, \in\, {\mathfrak k}$.

If $K\bare$ is isotropic, then taking into account that $[\fk\, , \fk]\,=\, \fk$, we
see that $\bilform (Z, b) = 0$ for all $Z \in \fk,$ and hence also for all $Z \in \fg = \fk^{\C}.$
It follows that $b = 0.$ Hence $c = a \in \ft \cap \fk$ and it follows that
the eigenvalues of $\ad (c)$ are all purely imaginary.

Conversely, assume that all eigenvalues of $\ad(c)$ are purely imaginary. Then $c \in \ft\cap \fk,$
so that $b = 0.$ It follows that $K$ is isotropic at the point $\bare.$ By invariance,
$K$ is isotropic everywhere.  This completes proof of the result involving $\Im \Omega.$ The proof
for $\Re \Omega$ is similar.
\end{proof}

\begin{lemma}\label{le3}
The $K$-orbit of $\bare\, =\, eZ(c)$ is Lagrangian with respect
to ${\rm Im}\,\Omega$ if and only if the form ${\rm Im}\, \Om$
is exact.
\end{lemma}

\begin{proof}
As in the proof of the previous lemma, it suffices to prove the assertion
with the word Lagrangian replaced by isotropic.

By Lemma \ref{le1},
the $G$-orbit  $G\bare \simeq G/L$ can be retracted onto the $K$-orbit
of $K\bare \simeq K/K\cap L$. Hence, the inclusion $K\bare \to G/Z(c)$ induces an isomorphism
on de Rham cohomology. Therefore, the closed form $\om \,=\,\text{Im}\,\Om$
is exact if and only if its restriction to $K \bare$ is exact.
Now the lemma
is a consequence of the following more general result.
\end{proof}

\begin{lemma}\label{le.re}
Let $K$ be a compact Lie group and $H\, \subset\, K$ a compact
subgroup containing a maximal torus of $K$. Let $\omega$ be a
$K$--invariant closed two--form on $K/H$. Then $\omega$ is
exact if and only if $\omega\, =\, 0$.
\end{lemma}

\begin{proof}
We need to show that if $\omega$ is exact, then
$\omega$ is identically zero.

Assume that $\omega\, =\, d\eta$. By integrating the left-translates $l_k^*\eta$ over $k \in K$
with respect to the Haar measure on $K$ of total volume $1$,
we may assume that the form $\eta$ is also $K$--invariant.

Let $T_0$ be a maximal torus of $K$ contained in $H$ and consider the natural fibration
$
\pi\, :\, K/T_0\, \longrightarrow\, K/H\, .
$
The pull back
$$
\widetilde{\eta}\, :=\, \pi^*\eta
$$
is a $K$--invariant one--form.
Let $\bare \, =\,
eT$. Then the evaluation $\widetilde{\eta}(\bare)$ is an
$\text{Ad}(T_0)$--invariant linear functional on the tangent
space $T_{\bare}(K/T_0)$. Its complex linear extension is therefore an
$\text{Ad}(T_0)$--invariant $\mathbb C$--linear
functional on the complexification
$T_{\bare}(K/T_0)\otimes_{\mathbb R} {\mathbb C}$.

A basis for this  complexification
is given by the canonical images of  root vectors $\{X_\alpha\}_{\alpha\in R}$,
where $R$ is a system of roots of $K^{\mathbb C}$ relative
to $T_0^{\mathbb C}$. The
$\text{Ad}(T_0)$--invariance of
$\widetilde{\eta}(\bare)$ implies that
$\widetilde{\eta}(\bare)\, =\,0$ on each of these root vectors, hence on $T_{\bare}(K/T_0).$
By $K$-invariance, it follows that $\widetilde{\eta}\, =\,
0$. Since $\pi$ is a surjective submersion, this in turn implies that
$\eta\, =\, 0$.
\end{proof}

Lemma \ref{le2} and Lemma \ref{le3} together complete
the proof of Theorem \ref{t2a}.
In view of Lemma \ref{l: real symplectic forms} this completes the proof of Theorem \ref{t2} (a).
For the remaining part of the proof of Theorem \ref{t2}, the following observation will be
of fundamental importance.

\begin{lemma}
\label{l: fibers isotropic}
The fibers of the fibration $\psi: G/Z(c) \to G/P$ are isotropic for the holomorphic
symplectic form $\Omega.$
\end{lemma}

\begin{proof}
Put $\bare : = e Z(c).$
By $G$-invariance, it suffices to show that $\Omega_{\bare}$ vanishes on the tangent
space at $\bare$ to the fiber $\psi^{-1}(eP) = PZ(c) = R_u(P)\bare.$
In view of (\ref{e: Omega at e}) and (\ref{e: eta and c}) it suffices to show that
$$
\bilform(c\,,\, [X,Y]) = 0
$$
for all $X,Y \in {\rm Lie}(R_u(P)).$ By linearity it suffices to prove
this identity for $X,Y$ contained in root spaces of $R_u(P).$
If $[X,Y] = 0,$ the identity is trivially valid, so we may assume $[X,Y] \neq 0.$
Then $[X, Y]$ is contained in a root space for a root $\ga$ of $P.$ Let $t \in T$ be such
that $t^\ga \neq 1.$ Then by $G$-invariance of $B,$
$$
B(c\,,\,[X,Y]) = B(\Ad(t^{-1})c\,,\,[X,Y]) = B(c\,,\,\Ad(t)[X,Y]) =
t^{\ga} B(c\,,\, [X,Y]).
$$
The lemma follows.
\end{proof}

The rest of the paper will be devoted to the proof of Theorem \ref{t2} (b), or rather its
generalization to the setting of real semisimple Lie algebras.
We will proceed under the assumption that all eigenvalues of $\ad(c)$ are real. The case with
all eigenvalues purely imaginary is treated similarly. Thus, $\Re \Omega$ is a real symplectic form
on $G/Z(c)$ and $K/Z_K(c)$ is a Lagrangian submanifold for this form.
Moreover, by Lemma \ref{l: fibers isotropic} the fibers of the fibration
$G/Z(c) \to G/P$ are Lagrangian for $\Re \Omega.$

In order to facilitate the comparison with the theory of real semisimple Lie algebras, we
make a few more remarks about the real Lie algebra $\fg_\R,$ see the proof of
Lemma \ref{l: real symplectic forms}. This algebra  has a real Killing form
which we denote by $B_\R.$

\begin{lemma}
As maps $\;\fg \times \fg \to \C,$ the Killing forms $B$ and $B_\R$ are related by $B_\R = 2 \Re B.$
\end{lemma}

\begin{proof}
Let $A: \fg \to \fg$ be a complex linear map. Its complex trace is denoted by $\Tr_\C A.$
At the same time $A$ defines a real linear endomorphism of $\fg_\R.$ As such, its trace is
denoted by $\Tr_\R A.$ It is straightforward to check that $\Tr_\R A = 2 \Re \Tr_\C A.$
Hence, for $X, Y \in \fg$ we have $B_\R (X,Y) = \Tr_\R (\ad(X)\circ \ad(Y)) =
2 \Re \Tr_\C (\ad(X)\circ \ad(Y)) = 2 \Re B(X,Y).$
\end{proof}

If $\gl \in \fg_\R^*$ we denote by $X_\gl$ the dual of $\gl$ relative to $B_\R,$
i.e.,  $\gl(Y) = B_\R(X_\gl, Y)$ for all $Y \in \fg_\R.$

\begin{lemma}
\label{l: compare c and X}
$c_\eta = 2 X_{\Re \eta}.$
\end{lemma}

\begin{proof}
For every $Y \in  \fg$ we have
$$
B_\R(2 X_{\Re \eta},Y) = 2 \Re \eta(Y) = 2 \Re B(c_\eta, Y) = B_\R(c_\eta, Y).
$$
The result now follows from the non-degeneracy of $B_\R.$
\end{proof}
We assumed that all eigenvalues of $c = c_\eta$ are real.
Because of Lemma
\ref{l: compare c and X}
it follows that the element $X_{\Re \gl}$ is real hyperbolic in the real semisimple Lie algebra
$\fg_\R,$ in the sense of Section \ref{sect4}. Let $B$ be a Borel subgroup of $G$
containing $T$ and such that the roots of $R_u(B)$ are non-negative on $c.$
Then the parabolic subgroup $P = Z(c)B$ corresponds to the parabolic subgroup $P(\Re \eta)$ introduced in
Section \ref{sect4}.
Therefore, the results of that section apply  to the present setting. In particular, the following
result is a special case of Theorem \ref{t: real Lagrangian fibration}.

\begin{theorem}
\label{t: complex Lagrangian fibration}
Let $\eta \in \fg^*$ be such that $c = c_\eta$ belongs to ${\rm Lie T}$ and such that
${\rm ad}\, (c)$ has real eigenvalues. Then the projection
$$
G/Z(c) \,\longrightarrow\, G/P
$$
is a Lagrangian fibration with Lagrangian section
$K/Z_K(\eta)$ relative to the symplectic form ${\rm Re}\, \Omega_{\eta}$.
Moreover, there exists a unique symplectic isomorphism from this fibration
onto the cotangent fibration $T^*(G/P)\,\longrightarrow\, G/P$ equipped
with the Liouville symplectic form, mapping $K/Z_K(\eta)$ to
the zero section.
\end{theorem}

Theorem \ref{t2} (b) follows from this result.

\section{Background in symplectic geometry}\label{sect3}
In this section we will discuss some background from symplectic geometry.
Let $M$ be a smooth manifold, and let $\pi\, :\,Z \,\longrightarrow\,M$ be a fiber bundle
whose total space $Z$ is equipped with a symplectic form
$\Omega.$ The bundle $\pi$ is called \textit{Lagrangian}
if for each point $x \,\in\, M$ the fiber $\pi^{-1}(x)$ is a Lagrangian submanifold of $Z$. A section
$$
s\,:\, M \,\longrightarrow\, Z
$$
is said to be \textit{Lagrangian} if the image $s(M)$ is a
Lagrangian submanifold of $Z$. If $\pi: Z \to M$ is Lagrangian, then by application
of the Darboux theorem, it follows that for any point $z_0 \in Z$
there exists a Lagrangian section $s$ of $Z$
locally defined in a neighborhood of $m_0 = \pi(z_0)$
and with $s(m_0) = z_0.$

The following result is well known in basic symplectic geometry and can be found in
\cite{ArnoldGivental}, Sect.\ 4.2. See also \cite{Weinstein}, where the result is established in the context
of Banach manifolds, with a useful review of the finite dimensional case.
A manifold $M$ will be identified with a submanifold of its cotangent bundle
$T^* M$ through the zero section.

\begin{theorem}\label{thb1}
Let $\pi\, :\,Z \,\longrightarrow\,M$ be a fiber bundle whose total space $Z$
is equipped with a symplectic form $\Omega.$
Assume that:
\begin{enumerate}
\item[{\rm (1)}]
$\pi$ has Lagrangian fibers;
\item[{\rm (2)}]
$\pi$ admits a Lagrangian section $s$.
\end{enumerate}
Let $p\, : \, T^*M \,\longrightarrow\,  M$ be the cotangent bundle of
$M$ equipped with the Liouville symplectic structure $\sigma$.
Then there exists an open neighborhood $U$ of $M$ in $T^*M$ and
embedding $\varphi\,:\, U \,\longrightarrow\, Z$
such that
\begin{enumerate}
\item[{\rm (a)}]
$\pi \circ \varphi \,= \, p$ on $U$;
\item[{\rm (b)}]
$\varphi  \,= \, s$ on $M$;
\item[{\rm (c)}]
$\varphi^*(\Omega) \,= \,\sigma$.
\end{enumerate}
If $\gf'\,:\, U' \,\longrightarrow\, T^*M$ is a second such embedding,
then $\gf' \,= \,\gf$ on an open neighborhood of $M$ in $U \cap U'$
\end{theorem}

Although this result is well known,
we include a proof to prepare for our later arguments leading to
the proof of Theorem \ref{t: real Lagrangian fibration}, see also
Theorem \ref{t: complex Lagrangian fibration}. The point is that
there is a canonical way to define the map $\gf.$

We agree to write $n$ for the dimension
of $M$. Then $s(M)$ is a submanifold of $Z$ of dimension $n$. Since this
submanifold is Lagrangian, the dimension of $Z$ must be $2n$.
The fibers of $\pi$ have dimension $n.$

Let $x \in M$ and $\eta \in T_x^*M.$ For each $z \in \pi^{-1}(x)$ we define
a vector $H_\eta(z) \in T_zZ$ by the requirement that
\begin{equation}
\label{e: defi H eta in general}
\Omega_z(X, H_\eta(z)) = \eta(d\pi(z)X),\qquad \forall X \in T_zZ.
\end{equation}
Since $d\pi(z) = 0$ on $T_z\pi^{-1}(x),$ which is a Lagrangian subspace
of $T_z Z,$ it follows that $H_\eta(z)$ belongs to this Lagrangian subspace.
Hence $H_\eta(z)$ is tangent to the fiber $\pi^{-1}(x)$ at any of its points $z.$
Accordingly, $H_\eta$ will be viewed as an element of $\Vect(\pi^{-1}(x)),$ the space
of vector fields on $\pi^{-1}(x).$

We will use the flows of these vector fields to define $\gf.$ The motivation
for the above definition is the following relation to Hamilton vector fields of functions
that are constant along the fibers of $\pi.$

\begin{lemma}
\label{l: H eta equals H f}
Let $x \in M,$ $\eta \in T_x^*M$ and let $\bar f: M \to \R$ be a smooth function such that
$d\bar f(x) = \eta.$ Let $f = \pi^*(\bar f)$ and let $H_f$ be the associated Hamilton vector field.
Then
$$
 H_f = H_\eta   \qquad \text{on}\quad \pi^{-1}(x).
$$
\end{lemma}

\begin{proof}
This is an immediate consequence of the definitions of $H_\eta$ and $H_f.$
\end{proof}

\begin{corollary}
\label{c: commuting vector fields in general}
Let $x \in M$ and $\eta_1, \eta_2 \in T_x^*M.$ Then $H_{\eta_1}$ and $H_{\eta_2}$ commute
as vector fields on the fiber $\pi^{-1}(x).$
\end{corollary}

\begin{proof}
We select smooth functions $\bar f_j: M \to \R$ with $d\bar f_j(x) = \eta_j$ and define $f_j = \pi^*(\bar f_j).$
Then $H_{f_1}f_2 = 0,$ hence $\{ f_1, f_2\} = 0$ and it follows that $H_{f_1}$ and $H_{f_2}$ commute.
These vector fields are tangent to the fiber $\pi^{-1}(x),$ hence their restrictions to the fiber commute.
These restrictions equal $H_{\eta_1}$ and $H_{\eta_2}$ by the lemma above.
\end{proof}
\medbreak\noindent
{\em Proof of Theorem \ref{thb1}.\ }
If $\eta \in T_x^*M$ we denote by $t \mapsto e^{tH_\eta}s(x)$ the integral curve
of $H_\eta$ in $\pi^{-1}(x)$ with starting point $s(x).$ Its maximal interval of definition
is denoted by $I_\eta.$  There exists an open neighborhood $U$ of $M$ in $T^*M$ such that
for each $x \in M,$ the open set $T^*_xM\cap U$ is star shaped and for each $\eta \in T_x^*\cap U$
the interval $I_\eta$ contains $(-2,2).$
We define $\gf: U \to Z$ by
\begin{equation}
\label{e: defi gf for Lagrange fibration}
\gf(\eta) = e^{H_\eta} s(p(\eta)), \qquad (\eta \in U).
\end{equation}
Then $\gf$ is a local diffeomorphism at each point of $M$ and coincides with an embedding on $M.$
Shrinking $U$ if necessary, we may arrange that in addition to the above, $\gf$ becomes a diffeomorphism from $U$
onto an open neighborhood of $s(M)$ in $Z.$ {}From the construction it is clear that (a) and (b) of
Theorem \ref{thb1} are satisfied.

We will now establish (c). As this is a local statement,
we may assume that there exists a diffeomorphism $\bar f = (\bar f_1, \ldots, \bar f_n)$ from $M$ onto an
open subset of $\R^n.$ Put $f = \pi^*(\bar f) = (f_1, \ldots, f_n).$ Then by
Lemma \ref{l: H eta equals H f} the Hamilton
vector fields $H_{f_i}$ are all tangent to the fibers of $\pi,$ from which we deduce that
$\Omega(H_{f_i}, H_{f_j}) = 0,$ for all $1 \leq i,j \leq n.$

Define $g: \gf(U) \to (\R^n)^*$ by
\begin{equation}
\label{e: defi g}
g(\gf(\eta)) = d\bar f(\pi(\eta))^{-1*}\eta,\qquad (\eta \in U).
\end{equation}
If $t = (t_1, \ldots, t_n) \in (\R^n)^*$ we agree to write  $t \bar f = t\after \bar f = t_1 \bar f_1 + \cdots + t_n \bar f_n$
and $\xi(x,t) =  d(t\bar f)(x) = t \after d\bar f(x) = t_1 d\bar f_1(x) + \cdots + t_n d\bar f_n(x).$
Then
$$
g_j(e^{t_1 H_{f_1}}\circ \cdots \circ e^{ t_n H_{f_n}} s(x)) =
g_j(e^{H_{\xi(x,t)}}s(x)) = \pr_j \,d\bar f(x)^{-1*}\, \xi(x,t) = t_j,
$$
for $(x,t)$ in a suitable neighborhood of the zero section in $M \times ( \R^n)^*.$
{}From this we see that $H_{f_i}g_j = \delta_{ij},$ so that $\Omega(H_{f_i}, H_{g_j}) = \gd_{ij}$
for all $1 \leq i,j \leq n.$ The functions $g_j$ are constant on the Lagrangian submanifold
$s(M)$ of $Z.$ Therefore, the vector fields $H_{g_j}$ are tangent to $s(M),$
and it follows that $\{g_i, g_j\} = \Omega(H_{g_i}, H_{g_j}) = 0$
on $s(M).$ Now $H_{f_k}\{g_i, g_j\} = \{ f_k, \{g_i, g_j\}\} = 0$ by application of the Jacobi identity.
It follows that $\Omega(H_{g_i}, H_{g_j}) = 0$ on a suitable neighborhood of $s(M)$ in $Z.$
We conclude that $\Omega = \sum_i df_i \wedge dg_j$ on this neighborhood,
by evaluation on the vector fields $H_{f_i}, H_{g_j}.$ Shrinking
$U$ if necessary, we may assume the identity to hold on $\gf(U).$
Hence, $\Omega|_{\gf(U)}$ is the pull-back under $(f,g)$ of the standard symplectic form on $\bar f(M) \times (\R^n)^*.$
Let $F: T^*M \to T^*\bar f(M) = \bar f(M) \times (\R^n)^*$ be the canonical symplectic isomorphism induced
by $\bar f.$ Then
for (c) it suffices to prove that the following diagram commutes
$$
\begin{array}{rcl}
T^*M \supset U & \!\!{\buildrel \gf \over \longrightarrow} \!\!& \gf(U) \subset Z\\
\scriptstyle{F}\; \searrow  & & \swarrow \scriptstyle{(f,g)}\\
&\bar f(M) \times (\R^n)^*&
\end{array}
$$
Let $\xi \in U$ and put $x = p(\xi).$ Then, by definition, $F(\xi) = (\bar f(x), d\bar f(x)^{-1*}\xi).$
On the other hand,
$$
(f, g)(\gf(\xi)) = (f(\pi(\gf(\xi)), d\bar f(\pi(\xi))^{-1*}) = (\bar f(x), d\bar f(x)^{-1*}\xi),
$$
by (\ref{e: defi g}),
and commutativity of the diagram follows.

It remains to establish uniqueness.
Assume that $\gf$ satisfies the conditions of the theorem.
We will show that it must be given by (\ref{e: defi gf for Lagrange fibration})
in a neighborhood of the zero section.
The cotangent bundle $p: T^*M \to M$ is Lagrangian, with $M$ as a Lagrangian section.
Hence, for $\eta \in T_x^*M$ and $\xi \in p^{-1}(x) = T^*M,$
we may define $\widetilde H_\eta(\xi) \in T_\xi(p^{-1}(x))$ as $H_\eta,$ but for the bundle $p$ instead of $\pi.$
Using that $\gf^*\Omega = \gs,$ it is an easy matter to check from the definitions that
$$
d\gf(\xi) \widetilde H_\eta(\xi) = H_\eta (\gf(\xi)),
$$
for all $\xi$ in a suitable neighborhood of $M$ in $T^*M.$ For the associated flows
in the fibers $p^{-1}(x)$ and $\pi^{-1}(x)$
this implies that
$$
\gf \after e^{t \widetilde H_\eta}  = e^{t H_\eta} \after \gf.
$$
A computation in local coordinates of $M$ shows that $e^{t \widetilde H_\eta}\xi = \xi + t \eta.$
On the other hand, $\gf(0_x) = s(x) = s(\pi(\eta)),$ and it follows that
$$
\gf(t\eta) = e^{t H_\eta} s(p(\eta)),
$$
for all $t$ in any interval containing zero  on which both expressions are well-defined.
It follows that $\gf$ must be given by (\ref{e: defi gf for Lagrange fibration}) on a suitable neighborhood of
$M$ in $T^*M.$
\qed

\section{Real semisimple groups}
In this section we recall some of the basic structure theory of
real semisimple Lie groups and their Lie algebras. As a basic reference
for this material we recommend \cite{Knapp}.

Let $G$ be a connected real semisimple group with finite center. The group
$G$ has a maximal compact subgroup $K.$ All such are conjugate and connected. The Killing
form $B$ of $\fg$ is known to be negative definite on $\fk$ and positive definite
on the orthocomplement $\fp$ of $\fk.$ In particular,
\begin{equation}
\label{e: Cartan deco}
{\mathfrak g}\, =\, {\mathfrak k}\oplus {\mathfrak p}\,
\end{equation}
as a direct sum of linear spaces. It is known that $[{\mathfrak k}\, ,{\mathfrak k}]\, \subset\,
{\mathfrak k}$, $[{\mathfrak k}\, ,{\mathfrak p}]\, \subset\,
{\mathfrak p}$ and $[{\mathfrak p}\, ,{\mathfrak p}]\, \subset\,
{\mathfrak k}.$
The decomposition (\ref{e: Cartan deco}) is called the Cartan decomposition of $\fg$ associated
with the maximal compact subgroup $K.$ It is readily seen that this decomposition
is $\Ad(K)$-invariant.

The map $\theta: \fg \to \fg$ given by $\theta = I$ on $\fk$
and $\theta = -I$ on $\fp$ is called the associated Cartan involution. It commutes with the adjoint action
of $K.$ We define the bilinear form
$\inp{\dotvar}{\dotvar}$ on $\fg$ by $\inp{X}{Y} = - B(X, \Cartan Y).$ Then $\inp{\dotvar}{\dotvar}$
is a $K$-invariant positive definite inner product on $\fg;$ in other words, $\Ad(K)$ acts
by orthogonal transformations with respect to it.
We note that $\ad \fp$ consists of symmetric transformations.

It is known that the map $(k, X) \mapsto k \exp X$ is a real analytic diffeomorphism
of $K \times \fp$ onto $G.$ Define $\widetilde \Cartan: G \to G$ by $\widetilde\Cartan(k \exp X) = k \exp(-X),$
then it is readily verified that $\widetilde \Cartan$ is an involution of $G$ with derivative
equal to the Cartan involution $\Cartan$ of $\fg.$ We agree to write $\Cartan$ for $\widetilde \Cartan;$
this involution of $G$ is also called the Cartan involution associated with $K.$

\begin{remark}
\label{r: comparison with G complex}
\rm
We note that if $G$ is a complex semisimple group, then it may be viewed as
a real semisimple Lie group with finite center. If $K$ is a maximal compact subgroup, then $\fp = \sqrt{-1}\cdot \fk,$ and $\theta$ is the involution
associated with the real form $\fk.$

On the other hand, if $G$ is linear,
 then $G$ has a complexification $G^\C$ and
$$
\widetilde{\mathfrak k}\, =\, {\mathfrak k}\oplus \sqrt{-1}\cdot
{\mathfrak p}
$$
is the Lie algebra of a maximal compact subgroup of $G^\C.$
\end{remark}

Let $\mathfrak a$ be a maximal abelian subspace of
$\mathfrak p$. It is known that all such are conjugate under $K.$
For each linear functional $\lambda \in \fa^*,$ we put
\begin{equation}
\label{e: eigenspace fa}
{\mathfrak g}_\lambda: = \{ X \in \fg \mid [H,X] = \gl(H) X,\quad \forall H \in \fa\}.
\end{equation}
Since $\fa$ is abelian, and $\ad(H)$ is  symmetric for $\inp{\dotvar}{\dotvar},$
for all $H \in \fa,$ the adjoint representation of $\fa$ in $\fg$ has a simultaneous diagonalization.
It follows that $\fg$ decomposes as a finite direct sum of joint eigenspaces of the form (\ref{e: eigenspace fa}).
Let $\gS$ be the set of nonzero $\gl \in \fa^*$ with $\fz(\gl) \neq 0.$ Then
$$
{\mathfrak g}\, =\, {\mathfrak g}_0\;\oplus\;
\bigoplus_{\ga \in \Sigma}\; {\mathfrak g}_\ga\, .
$$
It is known that $(\fa, \gS)$ is a root system, which is possibly non-reduced. A root $\ga \in \gS$
is called reduced if $\frac 12 \ga$ is not a root.  The set $\gS_0$ of all reduced roots forms a
genuine root system in $\fad.$ For each $\ga \in \gS$ there exists a unique $\ga_0 \in \gS_0$
such that $\ga \in \{\ga_0, 2 \ga_0\}.$ We note that $[\fg_\ga , \fg_\gb] \subset [\fg_{\ga + \gb}]$
for all $\ga,\gb \in \gS.$

A positive system for $\gS$ is a subset $\Pi$ of $\gS$ such that $\gS = \Pi \cup (-\Pi),$
and $\Pi $ and $- \Pi$ are separated by a hyperplane in $\fad,$ i.e., there exists a
$H \in \fa$ such that $\Pi = \{ \ga \in \gS \mid \ga(H) > 0 \}.$
It follows that $\Pi \mapsto \Pi_0 := \Pi\cap \gS_0$ defines a bijection
from the set of positive systems of $\gS$ onto the set of positive systems for $\gS_0.$
Let $\fa^{\reg}$ be the complement in $\fa$ of the union of all  root hyperplanes $\ker \ga$
for $\ga \in \gS.$ Then the connected components of $\fa^\reg$ are called the open Weyl chambers
of the root system $\gS.$ There is an obvious bijection between the set of all
such chambers and the set of positive systems for $\gS.$

Clearly, $\fg_0$ is the centralizer of $\fa$ in $\fg.$ As $\theta = -I$ on $\fa,$ it
follows that $\fg_0$ is invariant under $\theta.$ The centralizer of $\fa$ in $\fk$ is denoted by $\fm.$
Since $\fa$ is maximal abelian in $\fp,$ the
intersection $\fg_0 \cap \fp$ equals $\fa.$ Therefore,
$$
{\mathfrak g}_0\, =\,  {\mathfrak m}\, \oplus \,{\mathfrak a},
$$

\begin{remark}
\label{r: comparison roots}
\rm
In the notation of Remark \ref{r: comparison with G complex},
$\ft:= \sqrt{-1}\cdot\fa \oplus \fa$ is a maximal torus
of $\fg$ whose intersection with $\fk$ is a maximal torus of $\fk.$
Moreover, $\fa$ is the real subspace of $\ft$ consisting of all points
on which the roots of $\ft$ are real. Let $R$ be the set of $\ft$-roots. Then
restriction to $\fa$ induces an isomorphism $R \to \gS.$ In particular, the root
system $\gS$ is reduced in this setting. Accordingly, the root spaces for $\ft$ coincide with
those for $\fa.$ Finally, $\fg_0 = \ft$ and $\fm = \ft \cap \fk = \sqrt{-1}\cdot \fa.$
\end{remark}

Fix a positive system $\gS^+$ for $\gS.$   Let
${\mathfrak n}$ be the sum of all positive root spaces $\fg_\ga,$ for $\ga \in \gS^+.$
Since $\theta = - I$ on $\fa,$ we have
$$
\theta({\mathfrak g}_\ga)\, =\, {\mathfrak g}_{-\ga}\, ,
$$
for every $\ga \in \gS.$ Hence,
$$
{\mathfrak g}\, =\, \theta(\fn)\, \oplus \, {\mathfrak g}_0\,\oplus\, {\mathfrak n}\, .
$$
As ${\mathfrak k}$ is the eigenspace of $\theta$ for the
eigenvalue $1$, we see that
\begin{equation}
\label{e: fk and root spaces}
{\mathfrak k}\, =\, \fm \oplus \sum_{\ga\in \gS^+,\; X\in \fg_\ga} (X + \theta(X))\, .
\end{equation}
It now follows that
$$
\fg = \fk \oplus \fa \oplus \fn,
$$
as a direct sum of vector spaces. The exponential map $\exp: \fg \to G$ maps
$\fa$ and $\fn$ diffeomorphically onto closed subgroups $A$ and $N$ of $G,$ respectively.
Moreover, one has the so-called Iwasawa decomposition
\begin{equation}
\label{e: Iwasawa deco}
G = K A N,
\end{equation}
the multiplication map $(k, a, n)\mapsto kan$ being a diffeomorphism $K \times A \times N \to G.$

\section{Parabolic subgroups}
We recall that a Borel subalgebra of the complexified semisimple Lie algebra  $\fg^\C$
is by definition a maximal solvable subalgebra. A subalgebra of $\fg^\C$ which contains a Borel subalgebra
is said to be parabolic. It is well known that such a subalgebra equals
its own normalizer in $\fg^\C.$

A parabolic subalgebra of $\fg$ is defined to be a subalgebra  $\fP$ whose
complexification $\fP^\C$ is parabolic in $\fg^\C.$
Such an algebra $\fP$ equals its own normalizer in $\fg.$

A parabolic subgroup of $G$ is defined to be a subgroup $P$ which is the normalizer of
a parabolic subalgebra $\fP$ of $\fg.$ Being its own normalizer, $\fP$ is the Lie algebra of $P.$ We proceed by
describing the basic structure theory of parabolic subgroups of $G.$ Details can be
found in, e.g., \cite{Var}, p.\ 279.

The algebra $\fm + \fa + \fn = \fg_0 + \fn$
is a parabolic subalgebra of $\fg.$ It is known to be minimal in the sense that it does not contain any strictly smaller parabolic subalgebra.
The associated minimal parabolic subgroup of $G$ is given
by
$$
P_0 = M A N.
$$
Note that this decomposition is compatible with the Iwasawa decomposition (\ref{e: Iwasawa deco}).
In particular, the multiplication map $M \times A \times N \to P_0$ is a diffeomorphism.
Moreover, from the Iwasawa decomposition (\ref{e: Iwasawa deco}) it follows that
\begin{equation}
\label{e: G is K times P zero}
G = K P_0.
\end{equation}

It is known that every parabolic subgroup of $G$ is conjugate to one containing $P_0.$
The parabolic subgroups containing $P_0$ are finite in number, and may be described as follows.
Let
$$
\bfap := \{H \in \fa \mid \ga(H) \geq 0,\;\; \forall \ga \in \Pi\}
$$
be the closed positive Weyl chamber in $\fa.$
Given $c \in \bfap$ we define
$$
\fP(c) = \bigoplus_{\ga \in \gS, \, \ga(c) \geq 0}\; \fg_\ga.
$$
Clearly, this is a subalgebra of $\fg$ containing $\fm \oplus \fa \oplus \fn,$ hence parabolic.
Moreover, it depends on $c$ through the set $\Pi(c):= \{\ga \in \Pi \mid \ga(c) > 0\}.$
It can be shown that every parabolic subalgebra of $\fg$ containing $\fm \oplus \fa \oplus \fn$
is of the form $\fP(c)$ for some $c \in \bfap.$ In particular, we see that these parabolic
subalgebras are finite in number.

Let $\fn(c)$ be the sum of the root spaces $\fg_\ga$ for $\ga \in \Pi(c).$ Then
$$
\fP(c) = \fz(c) \oplus \fn(c).
$$
As $\fz(c)$ is reductive and normalizes the nilpotent subalgebra $\fn(c),$
this is a Levi decomposition of $\fP(c).$ In particular, $\fn(c)$ is the nilpotent radical of $\fP(c).$
Since $\fn(c) \subset \fn,$ the exponential map maps $\fn(c)$ diffeomorphically onto a closed subgroup
$N(c)$ of $G.$ Let $P(c)$ be the normalizer of $\fP(c)$ in $G,$ then we have the semi direct product
decomposition
$$
P(c) = Z(c) \ltimes  N(c).
$$
In particular, $N(c)$ is the unipotent radical of $P.$

Finally, we note that $\fP(c)$ is the sum of the eigenspaces  for the nonnegative
eigenvalues of $\ad(c).$ Now this definition can be given for any element $c \in \fg$ which
is {\it real hyperbolic}, i.e., for which $\ad(c)$ diagonalizes with real eigenvalues.
Moreover, $\Ad(x) \fP(c) = \fP(\Ad(x) c),$ for each $x \in G.$ It is known
that every real hyperbolic element is conjugate to an element of $\bfap.$ {}From what we just
said, it follows that every algebra of the form $\fP(c),$ with $c$ real hyperbolic, is a parabolic
subalgebra of $\fg.$ Moreover, since all minimal parabolic subalgebras are conjugate,
it follows that every parabolic subalgebra arises in this way.

We retain the assumption that $c \in \bfap.$ Let $P = P(c)$ be the associated parabolic subgroup of $G.$
Since $G/Z(c) \simeq G\times_{P} P/Z(c),$
the natural projection
$$
\pi: G/Z(c) \to G/P
$$
gives the quotient manifold
$G/Z(c)$ the structure of a real analytic fiber bundle over the real flag manifold $G/P,$
with fiber $P/Z(c)\simeq N(c).$

We note that $P$ contains $P_0,$ so that $G = K P,$ by (\ref{e: G is K times P zero}).
It follows that the inclusion map $K \to G$
induces a diffeomorphism $K /K \cap P \simeq G/P.$
Now $K \cap P = K \cap P \cap \Cartan P = K \cap Z(c).$ Put $Z_K(c) = K \cap Z(c).$ Then
\begin{equation}
\label{e: fibration}
G/Z(c) \simeq KP/Z(c) \simeq K \times_{Z_K(c)} P/Z(c) \simeq K \times_{Z_K(c)} N(c),
\end{equation}
exhibiting $G/Z(c)$ as a $K$-equivariant real analytic vector bundle over $K/Z_K(c) \simeq G/P.$

{}From (\ref{e: fk and root spaces}) it follows that the map $X \mapsto X + \Cartan X$
induces a linear isomorphism from $\fn(c)$ onto $\fk/\fk \cap \fz(c).$ This implies that
$K/Z_K(c)$ is a submanifold of $G/Z(c)$ of half the dimension.

\section{Real flag manifolds}\label{sect4}
Let $G$ be a connected real semisimple Lie group with
finite center, and let $\gl \,\in\, \fg^*$ be a real linear functional.
We write $X_\gl = B^{-1}(\gl),$ i.e.,
$$
 \gl(Y) \,= \, B(X_\gl\, , Y), \qquad (Y \in \fg).
$$
The element $\gl$ is called \textit{real hyperbolic} if
$\ad(X_\gl) \in \End(\fg)$ is diagonalizable with real eigenvalues.
{}From now on we assume $\gl$ to be real hyperbolic.

{}From the discussion in the previous section we know
that the element $X_\gl$ is conjugate to an element of the positive
chamber in $\fa.$ Thus, for the purpose of studying the
symplectic geometry of the coadjoint orbit through $\gl,$ we
may -- and will -- assume that $X_\gl$ is contained in the positive chamber
in $\fa$ from the start.

Let $\Omega \,=\, \Omega_\gl$
be the Kostant--Kirillov symplectic form on
the coadjoint orbit $G\cdot \gl$. The centralizer $Z(\gl)$ of $\gl$ in $G$ equals
$Z(X_\gl),$ by invariance of the Killing form.
Via
the natural
$G$--equivariant diffeomorphism
$$
G/Z(\gl) \,\longrightarrow\, G\cdot \gl\, ,
$$
the form
$\Omega$ may be pulled back to a symplectic form on $G/G(\gl)$.
For convenience, the latter form will also be denoted
by $\Omega$.

With notation as in the previous section, we write
${\mathfrak n}(\gl) = \fn(X_\gl)$ and $\fP(\gl) = \fP(X_\gl),$
and likewise $N(\gl) = N(X_\gl)$ and $P(\gl) = P(X_\gl).$
Then $P(\gl)$ is a parabolic subgroup of $G$ with Levi decomposition $P(\gl) = Z(\gl)N(\gl).$

The projection
\begin{equation}\label{bepi}
\pi\,:\, G/Z(\gl) \,\longrightarrow\, G/P(\gl)
\end{equation}
is a $G$--equivariant fibration with fibers  equal
to the $G$--translates of $N(\gl)\, \hookrightarrow\,
G/Z(\gl)$.
On the other hand, the cotangent bundle $T^*(G/P(\gl))$ comes
equipped with the natural Liouville symplectic form $\sigma$.

In this section we will prove the following result.

\begin{theorem}
\label{t: real Lagrangian fibration}
There exists a unique diffeomorphism
$$
\gf_\gl\,:\,  T^*(G/P(\gl)) \,\longrightarrow\, G/Z(\gl)
$$
satisfying the following conditions:
\begin{enumerate}
\item[{\rm (a)}]
$\pi \after \gf_\gl$ equals the projection of $T^*(G/P(\gl));$
\item[{\rm (b)}]
$\gf_\gl $ maps  the zero-section of $T^*(G/P(\gl))$ to $K/Z_K(\gl);$
\item[{\rm (c)}]
$\gf^*_\gl (\Omega_\gl) = (\sigma)$.
\end{enumerate}
\end{theorem}

The proof of this result is based on the ideas of the proof of Theorem \ref{thb1}
and will be given through a number of lemmas.
We start by observing that the symplectic form $\Omega$ on $G/Z(\gl)$
is $G$--invariant. At the origin ${\bar e} \,= \,eZ(\gl)$ it is
given by
\begin{equation}
\label{e: formula Omega at e}
\Omega_{\bar e}(X_\bare,Y_\bare ) = \gl([X,Y]) = B(X_\gl, [X,Y]) = -
B(X, [X_\gl, Y])\, ,
\end{equation}
for  $X,Y \in \fg.$

The following lemma expresses that we are in the set-up
of Theorem \ref{thb1}.

\begin{lemma}\mbox{\ }
\begin{enumerate}
\item[{\rm (a)}]
The fibers of $\pi,$ defined in \eqref{bepi},
are Lagrangian for $\Omega.$
\item[{\rm (b)}]
The submanifold $K/Z_K(\gl)\, \hookrightarrow\, G/Z(\gl)$ is
Lagrangian for $\Omega$.
\end{enumerate}
\end{lemma}

\begin{proof}
{}From (\ref{e: formula Omega at e}) one sees that $\Omega_{\bar
e}$
vanishes on
$\fn \times \fn,$ so that $T_{\bar e}(N(\gl)\bare)$ is
isotropic in $T_{\bar e}(G/Z(\gl)).$
We agree to write $\bar \fn = \Cartan\fn$ and $\bar \fn(\gl) = \Cartan \fn(\gl).$
{}From the decomposition
$$
\fg \,= \,{\bar \fn}(\gl) \oplus \fz(\gl) \oplus \fn(\gl)
$$
one sees  that $\dim \fg/\fz(\gl) \,= \,2\cdot \dim \fn$. Hence, the
orbit $N(\gl)\bare$ is Lagrangian in $G/Z(\gl).$
By equivariance, the fibers $gN(\gl)\bare$ are Lagrangian as well.
This establishes assertion (a).

For (b) we observe that
for $X\, ,Y \,\in \,\fk$ we have
\begin{eqnarray*}
\Omega_{\bar e}(X,Y) &= &B(X_\gl, [X,Y])
= B(\theta X_\gl, [\theta X, \theta Y])\\
& =&  -B(X_\gl, [X, Y]) = - \Omega_{\bar e}(X,Y).
\end{eqnarray*}
This implies that $T_{\bar e}(K/Z_K(\gl))$ is isotropic in
$T_{\bar e}(G/Z(\gl))$. By $K$-invariance, $K/Z_K(\gl)$ is isotropic
in $G/Z(\gl).$ In the text below equation (\ref{e: fibration}),
we observed that $K/Z_K(\gl)$ has dimension equal to half the dimension
of $G/Z(\gl).$
\end{proof}

We proceed by following the ideas of the proof of
Theorem \ref{thb1}.
Our first goal is to define suitable vector fields along the fibers of
the fibration (\ref{bepi}). The natural projection $G \to G/P(\gl)$ induces an isomorphism
from $\fg/\fP(\gl)$ onto the tangent space $T_{eP(\gl)}(G/P(\gl))$ which we use for identification of the two
spaces.

Let $\eta \in (\fg/\fP(\gl))^*$ and $n \in N(\gl).$ Since $\Omega_{n \bare}$
is non-degenerate on $T_{n \bare } G/G(\gl)$ we may define a tangent vector
$
H_{\gl, \eta}(n\bare) \,=\, H_\eta(n\bare) \,\in\, T_{n\bare}\,G/Z(\gl)
$
by
\begin{equation}\label{e: defi HV}
\Omega_{n\bare}(Z \,,\, H_\eta(n\bare)\,) \,=\, \eta(d\pi(n\bar e)Z)\,,
\end{equation}
for all $Z \in T_{n \bare}\, G/Z(\gl).$ Viewing $\eta$ as an element of $T_{eP(\gl)}^*(G/P(\gl)),$
we see that this tangent vector coincides with the
vector $H_\eta(n \bar e)$ defined in (\ref{e: defi H eta in general}). In particular, $H_\eta(n \bar e)$
is tangent to the fiber $N(\gl)\bar e.$

If $X$ is a homogeneous space for $G,$ and $g \in G,$
we denote by $l_g$ the left multiplication $x \mapsto gx$ on $X.$
For the space $G/P(\gl)$ we note that $dl_g(eP(\gl))$ is a linear isomorphism from
$\fg/\fP(\gl)$ onto $T_{gP(\gl)}(G/P(\gl)).$ Given $\eta \in (\fg/\fP(\gl))^*$ and $g \in G,$ we put
$$
g\cdot \eta : = dl_g(eP(\gl))^{-1*}\eta = \eta \after dl_g(eP(\gl))^{-1}.
$$

\begin{lemma}
\label{l: H eta and G action}
Let $g \in G$ and $\eta \in (\fg/\fP(\gl))^*.$
Let  $H_{g\cdot \eta} \in \Vect(gN(\gl)\bar e)$ be defined as in (\ref{e: defi H eta in general}),
for the bundle $\pi: G/Z(\gl) \to G/P(\gl).$
Then
$$
H_{g\cdot \eta}(g n \bar e)  = dl_g(\bar e) H_\eta(n \bar e), \qquad (n \in N(\gl)).
$$
\end{lemma}

\begin{proof}
This is a straightforward consequence of the $G$-invariance of $\Omega$ and the $G$-equivariance
of the projection map $\pi: G/Z(\gl) \to G/P(\gl).$
\end{proof}

At a later stage the situation that $g = m \in P(\gl)$ will be of particular interest to us.
As left translation by $P(\gl)$ fixes the element $eP(\gl)$ of $G/P(\gl),$
we see that $(m, \eta) \mapsto m\cdot \eta$
defines an action of $P(\gl)$ on $(\fg / \fP(\gl))^* \simeq T_{eP(\gl)}(G/P(\gl)).$

\begin{lemma}
\label{l: action induced by Ad}
The action $(m, \eta) \mapsto m \cdot \eta$ of $P(\gl)$ on $(\fg/\fP(\gl))^*$ is induced
by the adjoint action of $P(\gl)$ on $\fg.$
\end{lemma}

\proof
Let $\pr: G \to G/P(\gl)$ be the natural projection and
for $m \in P(\gl),$ let  $\cC_m: G \to G$ denote the conjugation map $x \mapsto m x m^{-1}.$ Then
$\pr \after \cC_m = l_m \after \pr.$ Differentiating this expression at the identity
element, we see that $dl_m(eP(\gl) \in \End(\fg/\fP(\gl))$ is induced by $d\cC_m(e) = \Ad(m).$
The result follows.
\qed

We will now derive a formula for the vector field $H_\eta$ along $N(\gl)\bar e$ that
will allow us to understand the global behavior of its flow.

\begin{lemma}
\label{l: H at e induces iso}
Let $n \in N(\gl).$ The map $\eta \mapsto H_\eta(n\bare)$ is a linear isomorphism from
$(\fg/\fP(\gl))^*$ onto $T_{n \bare}(N(\gl)\bar e).$
\end{lemma}

\proof
Since $\Omega_{n \bare}$ is non-degenerate, the map is injective. The expression on the right-hand
side of (\ref{e: defi HV}) vanishes for all $Z \in T_{n \bare}(N(\gl)\bare).$
Since $T_{n \bare}(N(\gl)\bare)$
is Lagrangian for the form $\Omega_{n \bare}$ it follows that $H_{\eta}(n\bare )$ belongs to
$T_{n \bare}(N(\gl)\bare).$ The result now follows for dimensional reasons.
\qed

For each point $n \in N(\gl)$ the natural map $n \mapsto n \bar e$ is an embedding
of $N(\gl)$ onto the closed submanifold $N(\gl)\bare = \pi^{-1}(eP(\gl))$ of $G/Z(\gl).$
The derivative of this embedding is a linear isomorphism from $T_n N(\gl)$ onto $T_{n\bar e} N(\gl)\bare$
through which we shall identify these spaces.

Since
 $\fn(\gl)$ and $\fP(\gl)$ are perpendicular for the Killing form $B,$
the map $X \mapsto - B(X, \dotvar)$ induces a linear map
\begin{equation}
\label{e: V to eta V}
V \mapsto \eta_V, \;\fn(\gl) \to (\fg/\fP(\gl))^*.
\end{equation}
As $B$ is non-degenerate, the map (\ref{e: V to eta V}) is a linear isomorphism onto.
Given $V \in \fn(\gl)$ we agree to write
$$
H_V(n) = H_{\eta_V}(n\bare), \qquad (n \in N(\gl)),
$$
viewed as an element of $T_n N(\gl).$ Accordingly, $H_V$ becomes a vector field on $N(\gl).$
The following lemma gives an explicit formula for this vector field.
It involves the endomorphism
\begin{equation}
\label{e: defi T gl}
T_\gl := \ad(X_\gl)|_{\fn(\gl)} \in \End(\fn(\gl)).
\end{equation}
As the roots of $\fn(\gl)$ are positive on $X_\gl,$
this endomorphism is invertible. {}From (\ref{e: formula Omega at e}) we see that
\begin{equation}
\label{e: Omega and T gl}
\Omega_\bare(X + \fz(\gl), Y + \fz(\gl)) = - B(X, T_\gl Y),
\end{equation}
for all $X,Y \in \fg.$

\begin{lemma}
Let $V \in \fn(\gl).$ Then for each $n \in N(\gl),$
\begin{equation}\label{e: direct formula HV}
 H_V(n) \,=\,  d l_n(e) \after T_\gl^{-1} \circ \Ad(n)^{-1}( V ).
\end{equation}
\end{lemma}
\medbreak
\vspace{1mm}
\begin{proof}
Let $\eta = \eta_V.$ {}From (\ref{e: defi HV}) it follows that for every $X \in \fg$ we have
\begin{eqnarray*}
\Omega_{\bar e}(X + \fP(\gl), H_\eta(\bar e))& = &  \eta(X + \fP(\gl))\\
&=& - B(V,X).
\end{eqnarray*}
On the other hand, since $H_\eta(\bar e) = H_V(e) + \fP(\gl)$ it follows from
(\ref{e: Omega and T gl}) that
$$
\Omega_{\bar e}(X + \fP(\gl), H_\eta(\bar e)) =  - B(X, T_\gl H_V(e)).
$$
Comparing the two equalities, and using that the Killing form is non-degenerate, we find that
\begin{equation}
\label{e: H V e}
H_V(e) = T_\gl^{-1} V.
\end{equation}
This establishes (\ref{e: direct formula HV}) for $n =e.$

To establish the formula in general, we observe that
from Lemma \ref{l: action induced by Ad} it follows that for every $n \in N(\gl)$ we have
$$
n \cdot [\eta_V] = \eta_V\after \Ad(n)^{-1} = \eta_{\Ad(n)V}.
$$
{}From Lemma \ref{l: H eta and G action} with $g = n^{-1}$ we now infer that
$$
H_{\eta_V}(n \bar e) = H_{\eta_{\Ad(n)^{-1}V}}(\bar e).
$$
Therefore,
$$
H_V(n) = dl_n(e) H_{\Ad(n)^{-1}V}(e) = dl_n(e) \after T_\gl^{-1}\after \Ad(n)^{-1}(V).
$$
\end{proof}

\begin{proposition}
\label{p: flow gives diffeo}
For every pair $V_1, V_2 \,\in\, \fn(\gl)$, the associated vector fields
$H_{V_1}$ and $H_{V_2}$
in $\Vect(N(\gl))$ commute. Moreover, the flows
of these vector fields are well defined as maps
$\R \times N(\gl)\, \longrightarrow\, N(\gl)$.
The associated map
$$
V \,\longmapsto\, \exp({H_{V}})e_N
$$
induces a diffeomorphism from $\fn(\gl)$ onto $N(\gl).$
\end{proposition}

\begin{proof}
The first assertion follows from Corollary \ref{c: commuting vector fields in general} applied
to the bundle $\pi: G/Z(\gl) \to G/P(\gl).$

Let now $V \,\in\, \fn(\gl)$. In our study of the flow of the vector field
$H_V \,\in\, \Vect(N(\gl))$ formula (\ref{e: direct formula HV}) will play
a crucial role.
Let $h_V$ denote the pull back of the vector field $H_V$
under the exponential map $\exp\,:\, \fn(\gl) \, \longrightarrow\, N(\gl).$
Thus,
\begin{equation}
\label{e: formula h V}
h_V(U) \,=\, d\exp(U)^{-1} H_V(\exp U).
\end{equation}
Now
\begin{equation}
\label{e: derivative exp}
d\exp(U) \,=\, dl_{\exp U}(e) \circ [I + R({\rm ad}\, U)],
\end{equation}
where $R$ is the analytic function ${\mathbb R}\, \longrightarrow\,
\mathbb R$ given by the convergent power series
$$
R(t) \,=\,  \frac{1 - e^{-t}}{t} - 1 = \sum_{n \geq 1}
\frac{(-t)^n}{(n+1)!}\, .
$$
Since $\fn(\gl)$ is nilpotent, there exists a smallest
positive integer $N_0$ such that $({\rm ad}\, U)^{N_0 + 1} \,= \,0$ for
all
$U \,\in\, \fn(\gl)$.
It follows that formula (\ref{e: derivative exp}) is also valid with the polynomial
$$
R(t) = \sum_{n =  1}^{N_0} \frac{(-t)^n}{(n+1)!}
$$
Combining (\ref{e: direct formula HV}), (\ref{e: formula h V}) and (\ref{e: derivative exp}),
we see
that the vector field $h_V$ on $\fn(\gl)$ is given by
\begin{eqnarray*}
h_V(U) & = & [I + R({\rm ad}\, U)]^{-1} \circ \,T_\gl^{-1} \circ  \,e^{ -
{\rm ad}\, U} ( V )\\
& = & V + \rho_\gl({\rm ad}\, U)(V),
\end{eqnarray*}
with $\rho_\gl \,=\, \rho \in {\mathbb R}[t]$ a polynomial
divisible by $t$; see (\ref{e: defi T gl}) for the definition of $T_\gl.$

We recall that, by assumption, the element ${\rm ad}\, X_\gl$
diagonalizes with real eigenvalues.
Let
$$
\nu_1 < \nu_2 < \cdots < \nu_p
$$
be the positive eigenvalues,
and let  $\fn(\nu_j)$ be the eigenspace associated with the eigenvalue
$\nu_j.$
Let $d_j$ be the dimension of this eigenspace. The sum
of the eigenspaces $\fn(\nu_j)$ equals $\fn(\gl).$ Accordingly, we fix a basis of eigenvectors
$V_1,\ldots, V_n$ for ${\rm ad}\, X_\gl$ in $\fn(\gl)$, and put $\fn_j\, =
\, {\mathbb R} V_j$.
By choosing a suitable numbering we may arrange that for each $k,$
$$
\bigoplus_{i \leq k}\;\; \fn(\nu_i) \,=\, \bigoplus_{j \leq d_1 + \cdots + d_k} \fn_j.
$$
Given $1 \leq k \leq n,$ we put
$$
\fn_{\geq k} = \bigoplus_{j \geq k}\;\; \fn_j\, .
$$
In addition, we put $\fn_{\geq k}\, =\, 0$ for $k \,>\, n$.
The subspaces $\fn_{>k}, \fn_{\leq k}$ and $\fn_{<k}$ of $\fn(\gl)$ are defined
in a similar fashion.

For $1 \leq j \leq n,$ we denote by $\pr_j$ the projection map $\fn(\gl)\,
\longrightarrow\, \fn_j$ along the remaining summands $\fn_i, \, i \neq j.$
Also, we define $\pr_{\geq k} = \sum_{j \geq k} \pr_j$. The projections
$\pr_{> k}$,
$\pr_{\leq k}$ and $\pr_{< k}$ are defined in a similar fashion.

By the Jacobi identity, ${\rm ad}\, X_\gl$ acts
on $[\fn(\nu_i), \fn(\nu_j)]$ as the scalar multiplication through
$\nu_i + \nu_j$. Hence, for each $k \,\geq\, 1$,
$$
[\fn(\gl)\, , \fn_k] \,\subset\, \fn_{> k}\, .
$$

Let $V \in \fn(\gl).$ Then the integral curve $t \,\longmapsto\, U(t)$ of the vector field $h_V$
with initial point
$U(0)\, =\, U_0$ is determined by the initial value problem
\begin{equation}\label{ebc}
U'(t) =  V + \rho({\rm ad}\, [U(t)]) (V)\, ,\qquad U(0) = U_0\, .
\end{equation}
For the component $U_1\, :=\, \pr_1\circ U$ in $\fn_1$ the equation
becomes
$$
U_1'(t)  \, =\, \pr_1 V\, ,\qquad U_1(0) = \pr_1 (U_0)\, .
$$
Indeed, $\rho({\rm ad}\, U) V$ has its values in $\fn_{>1}$.
The equation for $U_1$ has the solution
$$
U_1(t) \,=\, t \,\pr_1 V + \pr_1 (U_0)\, ,
$$
which is linear in $t$.
The remaining components may now be obtained by a recurrence
procedure and integration.
More precisely, let $k \geq 2,$ assume that $U_j := \pr_j \after U$ has been solved for
each $1 \leq j \leq k-1$, and put
$$
U_{<k} \, = \, \sum_{j < k} U_j\, .
$$
Then $U_k$ is determined by the initial value problem
$$
\left\{
\begin{array}{lll}
U_k'(t) & = & \pr_k V + \pr_k\circ \rho(\ad[U(t)])(V ) =
 \pr_k V + \pr_k \circ \rho(\ad[U_{< k}(t)]))( V) ,\\
U_k(0) & = & \pr_k (U_0).
\end{array}
\right.
$$
This equation may be solved directly by integration. By induction one sees
that the integral curve is defined for all $t\,\in\, \mathbb R$ and is
in fact a polynomial function of $t$.

For the final part of the proof it is important to make the following observation.
If $V \,\in\, \fn_{\geq k}$, then the integral curve $U(t)$ satisfies
\begin{equation}
\label{e: U t modulo gneq k}
U(t) - U_0 - t V \,\in\,  \fn_{> k}
\end{equation}
for all $t\, \in\, \mathbb R$.
Indeed, this follows by applying the projection $\pr_{\leq k}$ to the constituents of the equation
\eqref{ebc} and solving the resulting equation.

The vector fields $h_j \,:=\, h_{V_j}$, where $1 \leq j \leq n$,
form a collection of commuting vector fields
for which the flows are defined globally. Consequently, the associated flow maps $(t, c) \mapsto e^{t h_j}c$
are smooth maps $\R \times \fn(\gl) \to \fn(\gl).$ Let $1 \leq k \leq n.$ It follows
from (\ref{e: U t modulo gneq k})
that
\begin{equation}
\label{e: pr of flow hk}
\pr_{\leq k} \;e^{th_k} x = \pr_{\leq k} x +  t V_k,
\end{equation}
for all $x \in \fn(\gl)$ and $t \in \R.$
Since the vector fields $h_j$ commute, the map
$
\gf\,:\, {\mathbb R}^n \times \fn(\gl) \,\longrightarrow\, \fn(\gl)
$
given by
$$
\gf(t, x) \,=\, e^{t_1 h_{1}}\circ \cdots\circ e^{t_n h_{n}}(x)\, ,
$$
defines a smooth action of $(\R^n, +, 0)$ on $\fn(\gl).$ It follows by repeated application
of (\ref{e: pr of flow hk}) that
$$
\pr_k\; \gf(0, t_k, \ldots, t_n, x) = \pr_k x + t_k V_k,
$$
for all $t_k, \ldots, t_n \in \R$ and $x \in \fn(\gl).$
We will use this observation to show that the action $\gf$ is proper.

For the proof of this, it is convenient to have the following
notation. For a compact subset $\cC \,\subset\, \fn(\gl)$ we write
$$
T(\cC)\,:=\,
\pr_{\R^n}(\gf^{-1}(\cC) \cap \;{\mathbb R}^n \times \cC)\, =\,
\{ t\in {\mathbb R}^n \mid \exists\, c \in \cC:\; \gf(t, c) \in \cC\}.
$$
For proving properness of the action, it
suffices to show that for every compact subset $\cC\,\subset\, \fn(\gl)$,
the set $T(\cC)$ defined above is bounded.
Indeed, this implies that $\gf^{-1}(\cC) \cap \;{\mathbb R}^n
\times \cC$ is compact. For $1 \leq k \leq n,$ let $\pi_{\leq k}: \R^n \to \R^k$
be projection onto the first coordinates.
By induction on $k$ we will show that for every compact
set $\cC \,\subset\, \fn(\gl)$ the set $\pi_{\leq k} ( T(\cC))$ is bounded.

First, let $k \,=\, 1$, and let $t \in T(\cC)$. Then there exists
some $c\, \in \,\cC$
such that $\gf(t, c) \,\in\, \cC$. Since
$$
\pr_1 \;\gf(t , c) = \pr_1 \,c + t_1 V_1\, ,
$$
it follows that $t_1 V_1$ belongs to the vectorial sum $-\pr_1(\cC) + \pr_1 (\cC)$, which shows that
$\pi_{\leq 1}(T(\cC))$ is bounded.

Next, let $1 \,\leq\, k \,< \,n$, and assume that
$\pr_{\leq k}(T(\cC))$ is bounded
for every compact subset $\cC \, \subset\, \fn(\gl)$.  Let $t \,\in\,
T(\cC)$. Then there exists a $c \, \in\, \cC$
such that $\gf(t\, ,c) \, \in\, \cC$. The element $(t_1\, ,
\ldots\, , t_{k})$ lies in the  subset
$$
S\, := \, {\rm cl}\, \pr_{\leq k} (T(\cC))
$$
of ${\mathbb R}^k$, which is compact by the inductive hypothesis.
 It follows that
$$
\gf(0,t_{k+1}, \ldots, t_n, c) \,=\, e^{-t_1 h_1}\cdots e^{-t_{k}h_k}
\gf(t,c)
$$
lies in the compact subset
$$
\cC'=\{ e^{-t_1 h_1}\cdots e^{-t_k h_k}c\mid
(t,c) \in S \times \cC\}\, \subset\, \fn(\gl)\, .
$$
Now
$$
 \pr_{\leq k + 1} \;\gf(0,t_{k+1}, \ldots, t_n, c)\, =\,
 \pr_{\leq k+1} c  + t_{k+1} V_{k+1} \, ,
$$
{}from which we see that $t_{k+1} V_{k+1}$ belongs to
the vectorial sum $-  \pr_{\leq k+1} (\cC)  +  \pr_{\leq k+1} (\cC')$.
{}From this we conclude that $\pi_{\leq k+1} \,T(\cC)$ is bounded.

We now come to the final assertion. It follows from the above
that the map
$$
\psi\, :\, {\mathbb R}^n \,\longrightarrow\,  \fn(\gl)
$$
defined by $t\, \longmapsto\, \gf(t,0)$ is proper.
Moreover, since at every point the vector fields $h_1, \ldots, h_n$ are linearly
independent, it follows that $\psi$ is a local diffeomorphism.
Therefore, $\psi$ has open and closed image, hence it is surjective onto
$\fn(\gl)$. Moreover, the fibers of $\psi$ are finite and discrete. Hence,
$\psi$ is a covering map.
Since $\fn(\gl)$ is simply connected, it follows that $\psi$ is a diffeomorphism
from ${\mathbb R}^n$ onto $\fn(\gl)$. We consider the linear bijection
$$
\tau\,:\,  \fn(\gl) \,\longrightarrow\, {\mathbb R}^n
$$
given by $V \,= \, \sum_j \tau(V)_j V_j$.
By linearity of the map  $V \,\longmapsto \,h_V$
it follows that for all $V \,\in\, \fn(\gl)$,
$$
\exp(h_V)(0) = \exp(\tau(V_1) h_1 + \cdots + \tau(V_n) h_n) (0) =
\gf(\tau(V), 0)\, .
$$
This implies that the map defined by $V \, \longmapsto\, \exp(h_V)(0)$
is a diffeomorphism from $\fn(\gl)$ onto
$\fn(\gl)$. Since $h_V$ is the pull back of $H_V$ by the diffeomorphism
$\exp\,:\, \fn(\gl) \,\longrightarrow\,  N(\gl)$, it follows that
$$
e^{H_V} = \exp \circ \, e^{h_V} \circ \exp^{-1}.
$$
Hence, the map defined by $V \, \longmapsto\, e^{H_V}e_N$
is a  diffeomorphism from $\fn(\gl)$ onto $N(\gl)$. This completes
the proof of the proposition.
\end{proof}

\begin{corollary}
\label{c: defi Phi}
The map $\widetilde \Phi\, :\,  K \times (\fg/\fP(\gl))^* \, \longrightarrow\, G/Z(\gl)$
given by
$$
\widetilde \Phi(k, \eta) \,=\, k \exp(H_\eta)\bare
$$
induces a $K$--equivariant diffeomorphism of fiber bundles
$$
\Phi_\gl \, :\, K \times_{Z_K(\gl)} (\fg/\fP(\gl))^*
 \, \longrightarrow\,  G/Z(\gl)\, .
$$
Here the action of $Z_K(\gl)$ on $(\fg/\fP(\gl))^*$ is induced by the
adjoint action of $Z_K(\gl)$ on $\fg.$
\end{corollary}

\begin{proof}
The map $\widetilde \Phi$ is the composition of the
diffeomorphism
$$
\bar \Phi\, :\, K \times (\fg/\fP(\gl))^* \, \longrightarrow\, K \times N(\gl)\bare
$$
defined by $(k, \eta)\, \longmapsto (k, \exp (H_\eta)\bare)$ and the
submersion
$$
j\, :\, K \times N(\gl)\bare \, \longrightarrow\, G/Z(\gl)
$$
defined by $(k, n\bare)\,\longmapsto\, kn\bare$. The latter map
factors to a diffeomorphism $K \times_{Z_K(\gl)} N(\gl)\bare \, \longrightarrow\, G/Z(\gl)$.
The quotient $K \times_{Z_K(\gl)} N(\gl)$ is defined by using
the left action of $Z_K(\gl)$ on $N(\gl)\bare.$ It follows from Lemma \ref{l: H eta and G action}
that for all $m \in Z_K(\gl)$ we have
$$
l_m \circ \exp(H_\eta)\, =\,\exp(H_{m\cdot \eta}) \circ l_m
$$
on $N(\gl)\bare$. This implies that
$$
\bar\Phi(km, \eta) \,= \,\bar \Phi(k, m\cdot \eta)\, ,
$$
so that $\bar \Phi$ induces a  diffeomorphism
$K \times_{Z_K(\gl)} (\fg/\fP(\gl))^* \, \longrightarrow\, K \times_{Z_K(\gl)} N(\gl)\bare$.
In view of Lemma \ref{l: action induced by Ad} this completes the proof.
\end{proof}

The action of $K$ on $G/P(\gl)$ naturally induces an action of
$K$ on the total space of the cotangent bundle $T^*(G/P(\gl))$
through symplectomorphisms for
the Liouville symplectic form $\sigma$ on $T^*(G/P(\gl))$.
In particular, the stabilizer $Z_K(\gl) = K \cap P(\gl)$ acts linearly
on the cotangent space  of $T_{eP(\gl)} G/P(\gl)$ at $eP(\gl).$ The latter is
naturally identified with $(\fg/\fP(\gl))^*$. By Lemma \ref{l: action induced by Ad}
the resulting action of $Z_K(\gl)$ on $(\fg/\fP(\gl))^*$ coincides with the one
induced by the adjoint action of $Z_K(\gl)$ on $\fg.$

The map defining the action of $K$ on $T^*(G/P(\gl))$
induces a submersion
$$
K \times (\fg/\fP(\gl))^* \,\longrightarrow\, T^*(G/P(\gl))
$$
which factors to an isomorphism of vector bundles
$$
\Psi_\gl: \;\;K \times_{Z_K(\gl)} (\fg/\fP(\gl))^*
\,{\buildrel \simeq\over \longrightarrow} \, T^*(G/P(\gl))\, .
$$
We will complete the proof
of Theorem \ref{t: real Lagrangian fibration}
by showing that the bundle isomorphism
$$
\gf_\gl : = \Phi_\gl \after \Psi_\gl^{-1}: T^*(G/P(\gl)) \to G/Z(\gl)
$$
satisfies the properties
of the theorem.

{\em Completion of the proof of Theorem \ref{t: real Lagrangian fibration}.\ }
Since $K \cap P(\gl) = Z_K(\gl)$ and $G = KP(\gl),$ the inclusion map $K \to G$
induces a diffeomorphism $K/Z_K(\gl) \to G/P(\gl),$ whose inverse will be denoted by $\bar s.$
Let $j: K/Z_K(\gl) \to G/Z(\gl)$ be the embedding induced by the inclusion map; then $s = j \after \bar s$
is a section of the bundle $\pi: G/Z(\gl) \to G/P(\gl).$

Let $x \in G/P(\gl)$ and $\xi \in T_x^*(G/P(\gl)).$ Fix $k \in K$ such that $k Z_K(\gl) = s(x)$
and define $\eta: = dl_k(eP(\gl))^*\xi.$
Then
$$
\Psi_\gl([k, \eta]) = k \cdot \eta = \xi,
$$
so that
\begin{eqnarray*}
\gf_\gl(\xi) &=& \Phi_\gl([k, \eta])\\
&=& k \exp (H_\eta) \bar e\\
&=& \exp (H_{k\cdot \eta}) k \bar e\\
&=& \exp (H_\xi) s(x).
\end{eqnarray*}
It follows that $\gf = \gf_\gl$ equals the map defined by (\ref{e: defi gf for Lagrange fibration}), for the bundle
$\pi: G/Z(\gl) \to G/P.$ Moreover, in view of Lemma \ref{l: H eta and G action}, Proposition
\ref{p: flow gives diffeo} and Corollary \ref{c: defi Phi},
the proof of Theorem \ref{thb1} works with $U = T^*(G/P(\gl));$
in particular, all appearing flows of vector fields are defined without any restriction
on their domains.
This establishes conditions (a),(b) and (c) of Theorem \ref{thb1} with $U = T^*M$ and $\gf(U) = G/Z(\gl).$
{}From these, conditions (a), (b) and (c) of Theorem \ref{t: real Lagrangian fibration} follow.
Uniqueness of $\gf_\gl$ follow by the arguments of the proof of Theorem \ref{thb1} that are valid
without any restrictions on domains.
\qed


\end{document}